\documentclass[a4paper,reqno,11pt]{amsart}

\usepackage{setspace, amssymb, amsmath, amsthm, graphicx, color, comment, booktabs, multirow, hyperref, mathtools, thmtools, thm-restate}
\usepackage[T1]{fontenc}
\usepackage[utf8]{inputenc}
\usepackage[inline]{enumitem}
\usepackage[dvipsnames,svgnames,table]{xcolor}
\usepackage[foot]{amsaddr}
\usepackage[capitalise, noabbrev]{cleveref}

\usepackage[longnamesfirst,numbers,sort&compress]{natbib}

\hypersetup{
    pdftitle={Product structure of graphs excluding a topological minor},
    pdfauthor={Jędrzej Hodor, Hoang La, Piotr Micek, Clément Rambaud},
    colorlinks,
    linkcolor={RoyalBlue},
    citecolor={RubineRed},
    urlcolor={blue!80!black},
    nesting=false,
}

\colorlet{defcolor}{ForestGreen}
\newcommand{\defin}[1]{\relax\ifmmode{\color{defcolor}{#1}}\else{\emph{\textcolor{defcolor}{#1}}}\fi}

\newcommand{\td}{\operatorname{td}}

\newcommand{\tw}{\operatorname{tw}}
\newcommand{\utw}{\operatorname{utw}}

\newcommand{\Oh}{\mathcal{O}}

\newcommand{\calC}{\mathcal{C}}

\newcommand{\bigO}{\mathcal{O}}

\let\leq\leqslant
\let\geq\geqslant

\let\subset\subseteq

\let\epsilon\varepsilon
\let\phi\varphi

\DeclareMathOperator\clos{clos}

\renewcommand{\setminus}{-}

\usepackage{tikz}
\usetikzlibrary{arrows}
\usetikzlibrary{calc}
\tikzset{myspacing/.style = {outer sep = 5pt, inner sep = 5pt}}
\tikzset{implies/.style = {-implies, double distance=3pt, thick}}
\tikzset{box/.style = {rectangle, draw, myspacing, align=center}}

\providecommand{\noopsort}[1]{}

\makeatletter
\def\thm@space@setup{
\thm@preskip=4mm
\thm@postskip=0mm
}
\makeatother

\crefformat{equation}{#2(#1)#3}
\let\eqref\cref
\crefformat{subsection}{Subsection #2#1#3}
\crefformat{subsubsection}{Subsubsection #2#1#3}

\theoremstyle{plain}
\newtheorem{thm}{Theorem}
\newtheorem*{thm*}{Theorem}

\newtheorem{lemma}[thm]{Lemma}
\newtheorem*{lemma*}{Lemma}
\newtheorem{cor}[thm]{Corollary}
\newtheorem*{cor*}{Corollary}
\newtheorem{obs}[thm]{Observation}
\newtheorem{conj}[thm]{Conjecture}
\newtheorem*{conj*}{Conjecture}

\theoremstyle{remark}

\newtheorem*{problem*}{Open problem}

\newtheorem{claim}[thm]{Claim}
\newtheorem*{claim*}{Claim}

\crefname{obs}{Observation}{Observations}
\theoremstyle{definition}

\crefname{lem}{Lemma}{Lemmas}
\crefname{thm}{Theorem}{Theorems}
\crefname{cor}{Corollary}{Corollaries}

\newenvironment{proofclaim}[1][]
    {\let\oldqed\qedsymbol\renewcommand{\qedsymbol}{\ensuremath{\lozenge}}\begin{proof}[Proof of the claim] }{\end{proof}\renewcommand{\qedsymbol}{\oldqed}}

\setenumerate{label=\textup{(\roman*)}, noitemsep, topsep=3pt-\parskip, 
labelindent=.2em, leftmargin=*, widest=iii,}
\setitemize{noitemsep, topsep=-\parskip, labelindent=.2em, leftmargin=*, widest=iii,}

\newdimen\mywidth
\sbox0{m}%
\newdimen\mywidthprim
\sbox0{m}%
\newdimen\mywidthprimprim
\sbox0{m}%
\newdimen\mywidthA
\sbox0{m}%
\newdimen\mywidthAprim
\sbox0{m}%
\newcommand{\leaves}[1]{\mathrm{Leaves}(#1)}

\title{Product structure of graphs excluding a topological minor}

\begin{document}

\author[Hodor]{Jędrzej Hodor}
\address[J.~Hodor]{Theoretical Computer Science Department, 
Faculty of Mathematics and Computer Science and Doctoral School of Exact and Natural Sciences, Jagiellonian University, Krak\'ow, Poland}
\email{jedrzej.hodor@gmail.com}

\author[La]{Hoang La}
\address[H.~La]{LISN, Universit\'e Paris-Saclay, CNRS, Gif-sur-Yvette, France}
\email{hoang.la.research@gmail.com}

\author[Micek]{Piotr Micek}
\address[P.~Micek]{Theoretical Computer Science Department, 
Faculty of Mathematics and Computer Science, Jagiellonian University, Krak\'ow, Poland}
\email{piotr.micek@uj.edu.pl}

\author[Rambaud]{Clément Rambaud}
\address[C.~Rambaud]{Universit\'e C\^ote d'Azur, CNRS, Inria, I3S, Sophia-Antipolis, France}
\email{clement.rambaud@normalesup.org}

\thanks{This research was funded by the National Science Center, Poland under grant UMO-2023/05/Y/ST6/00079 within the WEAVE-UNISONO program. 
Additionally, J.\ Hodor was supported by a Polish Ministry of Education and Science grant (Perły Nauki; PN/01/0265/2022); H.\ La benefited from the support of the FMJH Program PGMO}

\begin{abstract}
    We prove that, for all positive integers $h$ and $t$ and every graph $X$ with $\td(X) \leq h$, 
    there exists a positive integer $c(X,t)$ such that every graph $G$ with $\tw(G) < t$ 
    that excludes $X$ as a topological minor is isomorphic to a subgraph of $H \boxtimes K_{c(X,t)}$ for some graph $H$ with $\tw(H) < 2^{h+1}-1$.
    This extends a result by Ding and Oporowski (Journal of Graph Theory; 1995),
    which states that 
    for all positive integers $\Delta$ and $t$, 
    there exists a positive integer $f(\Delta,t)$ such that 
    every graph $G$ with $\tw(G)<t$ and $\Delta(G)\leq\Delta$ is isomorphic to a subgraph of $T \boxtimes K_{f(\Delta,t)}$ for some tree $T$. 
\end{abstract}

\maketitle

\section{Introduction}\label{sec:intro}

In 1995, Ding and Oporowski~\cite{Ding_1995} proved that 
for all positive integers $\Delta$ and $t$, 
there exists a positive integer $f(\Delta,t)$ such that for every graph $G$ with treewidth less than $t$ and $\Delta(G) \leq \Delta$, 
there is a tree $T$ such that $G$ is isomorphic to a subgraph of $T\boxtimes K_{f(\Delta,t)}$. 
Here, for a graph $H$ and a nonnegative integer $c$,
we denote by $H \boxtimes K_c$ the graph obtained from $H$ by replacing every
vertex by a copy of $K_c$, and every edge by a copy of $K_{c,c}$.\footnotemark

The condition $\Delta(G) \leq \Delta$ is equivalent to excluding the star with $\Delta+1$ leaves as a topological minor. 
Our main result extends the theorem of Ding and Oporowski~\cite{Ding_1995} from stars to arbitrary fixed excluded topological minors. 
Although the first factor of the product is no longer a tree, its treewidth is bounded in terms of the treedepth of $X$. 
For a graph $G$, we denote its treewidth and treedepth by
\defin{$\tw(G)$} and \defin{$\td(G)$}, respectively.
These and all other necessary definitions are given in \cref{sec:preliminaries}.

\begin{thm}\label{thm:main}
    Let $h$ and $t$ be positive integers,
    and let $X$ be a graph with $\td(X) \leq h$.
    There exists a positive integer $c(X,t)$ such that,
    for every graph $G$ with $\tw(G)<t$,
    if $X$ is not a topological minor of $G$,
    then there exists a graph $H$ such that
    \begin{enumerate}
        \item $\tw(H) < 2^{h+1}-1$ and
        \item $G \subseteq H \boxtimes K_{c(X,t)}$.
    \end{enumerate}
\end{thm}

The significance of \Cref{thm:main} is best expressed in terms of the so-called \emph{underlying treewidth}. 
The \defin{underlying treewidth} of a class of graphs $\mathcal{C}$,
denoted by \defin{$\utw(\mathcal{C})$},
is the smallest integer $k$, if it exists, such that for every nonnegative integer $t$, there exists a positive integer $c(t)$ such that every graph $G \in \mathcal{C}$ with $\tw(G)<t$ is a subgraph of $H \boxtimes K_{c(t)}$ for some graph $H$ with $\tw(H)\leq k$.
If no such integer exists, then we say that $\mathcal{C}$ has unbounded underlying treewidth and write $\utw(\mathcal{C}) = +\infty$.

\footnotetext{More generally, for two graphs $A$ and $B$,
the \defin{strong product} of $A$ and $B$, denoted by \defin{$A \boxtimes B$}, is the graph with vertex set $V(A) \times V(B)$, and two distinct vertices $(a,b)$ and $(a',b')$ are adjacent when $aa' \in E(A)$ or $a=a'$, and $bb' \in E(B)$ or $b=b'$.}

Campbell, Clinch, Distel, Gollin, Hendrey, Hickingbotham, Huynh, Illingworth, Tamitegama, Tan, and Wood~\cite[Theorem~28]{CCDGHHHITTW22} proved that 
a monotone\footnote{A class of graphs $\mathcal{C}$ is \defin{monotone} if it is subgraph-closed: for every graph $G \in \mathcal{C}$, for every subgraph $H$ of $G$, $H \in \mathcal{C}$.} class of graphs $\mathcal{C}$ has bounded underlying treewidth if and only if all graphs in $\calC$ exclude some fixed graph as a topological minor. 
In particular, they showed that, for every integer $k$ with $k \geq 2$, the underlying treewidth of $K_k$-topological-minor-free graphs is $k-2$ if $k \in \{2,3,4\}$ and $k$ if $k \geq 5$. Therefore, for every fixed graph $X$, the class of graphs excluding $X$ as a topological minor has underlying treewidth at most $|V(X)|$. 
A stronger result is known if we exclude a graph as a minor rather than a topological minor. 
Dujmović, Hickingbotham, Hodor, Joret, La, Micek, Morin, Rambaud, and Wood~\cite[Theorem~2]{DHHJLMMRW24} 
proved that, if $X$ is excluded as a minor, then the treewidth of the
underlying factor can be bounded solely in terms of $\td(X)$. 
Thus, topological-minor
exclusion gave a bound depending on $|V(X)|$, while minor exclusion
gave a bound depending only on $\td(X)$.
\Cref{thm:main} combines these two features by excluding $X$ as a topological minor, while bounding the treewidth of the underlying factor solely in terms of $\td(X)$.

The lower-bound construction of~\cite[Lemmas~12 and~27]{CCDGHHHITTW22} also gives a converse estimate.
For every monotone class of graphs $\mathcal{C}$ and every positive integer $h$,
if every graph of treedepth less than $h$ is a topological minor of some graph in $\mathcal{C}$,
then $\utw(\mathcal{C}) \geq h-2$.
Combining this lower bound with \Cref{thm:main} yields the following.
\begin{cor}
    Let $\mathcal{C}$ be a monotone class of graphs with bounded underlying treewidth.
    Then $\mathcal{C}$ excludes some graph as a topological minor.
    Moreover, if $h$ is the least integer such that $\mathcal{C}$ excludes a graph of treedepth $h$ as a topological minor,
    then
    \[
        h-2 \leq \utw(\mathcal{C}) < 2^{h+1}-1.
    \]
\end{cor}

Thus, the underlying treewidth of a monotone class of graphs $\mathcal{C}$ is tied to
the minimum treedepth of a graph excluded as a topological minor in $\mathcal{C}$.

In the special case of $X=K_{a,b}$, 
we refine the proof of~\cref{thm:main} to show that the underlying treewidth is at most $a+4$, improving on the bound bound from \Cref{thm:main}, which is exponential in $a$. 
\begin{thm}\label{thm:main_Krs}
    Let $a$, $b$, and $t$ be positive integers with $a \leq b$.
    There exists a positive integer $c(a,b,t)$ such that,
    for every graph $G$ with $\tw(G)<t$,
    if $K_{a,b}$ is not a topological minor of $G$,
    then there exists a graph $H$ such that
    \begin{enumerate}
        \item $\tw(H)<a+5$ and
        \item $G \subseteq H \boxtimes K_{c(a,b,t)}$.
    \end{enumerate}
\end{thm}
Combined with the lower-bound construction of~\cite{CCDGHHHITTW22}, this determines the underlying treewidth up to an additive constant. 
Indeed, if $a$ and $b$ are positive integers with $a\leq b$ and $\mathcal{C}$ denotes the class of all graphs excluding $K_{a,b}$ as a topological minor, then
\[
a-1 \leq \utw(\mathcal{C}) \leq a+4.
\]
Before our work, the only known upper bound on $\utw(\mathcal{C})$ for $a\geq 4$ was a function that grew with $a$ and $b$, see~\cite[Theorem~24]{CCDGHHHITTW22} and the discussion after~\cite[Theorem~25]{CCDGHHHITTW22}.
See also Jacob, Lochet, and Paul~\cite{Jacob2025} for related structural results on graphs excluding simple topological minors, including $K_{2,b}$. 
We highlight that
\Cref{thm:main_Krs} is very specific to the topological minor relation.
In particular,
Liu and Wood~\cite[Theorem~12]{LiuWood2024} gave a construction showing that
$K_{2,b}$-subgraph-free graphs of treewidth (and even treedepth) at most $2k$
have unbounded underlying treewidth.
As an aside, the same authors showed in an earlier paper~\cite[Theorem~4]{LiuWood2019} 
that $K_{a,b}$-subgraph-free graphs of bounded treewidth
have clustered chromatic number at most $a+1$: 
for every $K_{a,b}$-subgraph-free graph $G$ of treewidth less than $t$,
there exists $H$ with $\chi(H) \leq a+1$ such that $G \subseteq H \boxtimes K_{f(a,b,t)}$, for some function $f$ (see also \cite{Liu2024}).

Our results are motivated by the emerging theory of \emph{product structure} of graphs.
A central result in this area, proved in 2019 by
Dujmovi\'{c}, Joret, Micek, Morin, Ueckerdt, and Wood~\cite[Theorem~36(a)]{DJMMUW20}, states that every planar graph $G$ is isomorphic to a subgraph of $H \boxtimes P$ for some path $P$ and some graph $H$ with $\tw(H)\leq 8$.
This theorem and its variants have had a significant impact on structural graph theory.
In particular, they have been used to show that
planar graphs have bounded queue-number~\cite{DJMMUW20} and bounded non-repetitive chromatic number~\cite{Dujmovic2020},
that the $p$-centered chromatic number of planar graphs is $\bigO(p^3\log p)$~\cite[Theorem~1]{Debski2021}, and that
$n$-vertex planar graphs admit an adjacency labeling scheme with labels of length $\log_2 n +o(\log_2 n)$~\cite[Theorem~1]{DEGJMM21}.
For an introduction to the planar product structure theorem and its applications, see~\cite{Dvorak2021}.
Several other product structure theorems have since been established.
In particular, Illingworth, Scott, and Wood~\cite[Theorem~4]{ISW22} gave an elegant product structure theorem for $K_t$-minor-free graphs that, in particular, bounds their underlying treewidth, while Dujmovi\'c, Hickingbotham, Joret, Micek, Morin, and Wood~\cite[Theorem~1]{Dujmovi2023} established a product structure theorem based on pathwidth rather than treewidth.
Both results were important sources of inspiration for this work.

Our results also have direct applications to $p$-centered colorings. 
For a graph $G$ and a positive integer $p$, we denote by \defin{$\chi_p(G)$} the $p$-centered chromatic number of $G$. See the definition in~\cref{sec:conclu}. 
Pilipczuk and Siebertz~\cite[Lemma~15]{PS21} proved a general bound $\chi_p(G) = \mathcal{O}(p^{\tw(G)})$. 
Dębski, Felsner, Micek, and Schr{\" o}der~\cite[Theorem 4]{Debski2021} proved that if  
$G$ excludes $X$ as a topological minor, then 
$\chi_p(G)\in\mathcal{O}(p^{f(X)})$ for some function $f$.
Fix positive integers $a$ and $b$ with $a\leq b$, and let $\mathcal{C}$ denote a class of graphs excluding $K_{a,b}$ as a topological minor. 
For graphs in $\mathcal{C}$ of bounded treewidth, our result gives an
exponent depending only on $a$. 
Indeed, let $c(a,b,t)$ be as in~\cref{thm:main_Krs}. 
Fix $G\in \mathcal{C}$ with $\tw(G)<t$, and let $H$ be the graph given by~\cref{thm:main_Krs}. Then
\[
\chi_p(G) \leq \chi_p(H\boxtimes K_{c(a,b,t)}) \leq c(a,b,t) \cdot \chi_p(H) = \Oh_{a,b,t}(p^{a+4}).
\]
We suspect that such a bound on $\chi_p(G)$ holds for all $G\in\mathcal{C}$ without the extra assumption on the treewidth of $G$. 
See \Cref{conj:cen_col}.

Our proof builds on the notion of ``rich models'', introduced and developed in~\cite{wcol_paper,hodor2026centeredcoloringsweakcoloring}.
Since these works concern minors rather than topological minors, the techniques must be adapted to the present setting.

The proof of \Cref{thm:main} proceeds by induction on the treedepth of the
excluded topological minor.
Let $\mathcal{F}$ be a family of connected subgraphs of a graph $G$.
A standard packing-covering lemma states that, for every nonnegative integer $d$, if $\mathcal{F}$ contains no $d+1$ pairwise vertex-disjoint members, 
then there is a set $Z\subseteq V(G)$ of size at most $(\tw(G)+1)\cdot d$ that intersects every member of $\mathcal{F}$.  
Moreover, paying a factor of two in the size of $Z$, we can get an additional property that every component $C$ of $G-Z$ has a controlled number of neighbors in $G$, 
i.e.\ $|N_G(V(C))|\leq 2(\tw(G)+1)$. 
See~\cref{lemma:helly_with_LCA}.
We prove a structured analogue of this lemma.
Instead of forbidding $d+1$ pairwise disjoint members of $\mathcal{F}$, we forbid a subdivision of a graph $X$ together with suitably connected members of $\mathcal{F}$.
We encode such configurations as $(T,\mathcal{F})$-structures for a rooted tree $T$ witnessing the treedepth of $X$. 
See \Cref{sec:preliminaries} for precise definitions.
Our main technical statement establishes the following:
if there is no such $(T, \mathcal{F})$-structure in $G$,
then there is a set $S \subseteq V(G)$ that intersects every member of $\mathcal{F}$,
and such that, while $S$ might be large, it has a suitable product structure.
See \Cref{thm:main_induction}.
When $\mathcal{F}$ consists of all one-vertex subgraphs of $G$, 
every such hitting set is equal to $V(G)$,
and hence the whole graph $G$ has the desired product structure.
The proof of this statement will be by induction on the
treedepth of the excluded topological minor,
and the possibility to choose the family $\mathcal{F}$
on which we apply the induction hypothesis is at the core of our argument.

In \Cref{sec:preliminaries}, we introduce the necessary definitions and establish several basic lemmas.
In \Cref{sec:colorful}, we prove a technical result, namely \Cref{lemma:colorful_lemma}, which is then used in \Cref{sec:main_proof} to prove \Cref{thm:main}.
In \Cref{sec:Krs}, we prove \Cref{thm:main_Krs}.
Finally, we conclude in \Cref{sec:conclu} with an open problem.

\section{Preliminaries}\label{sec:preliminaries}

For a positive integer $k$, we write $\defin{[k]}=\{1,\ldots,k\}$ and $\defin{[0]} = \emptyset$.
All graphs considered in this paper are finite, simple, and undirected.
The \defin{null graph} is the graph with no vertices.  
Connected graphs are nonnull.
A tree is defined as a connected forest, thus, trees and subtrees are also assumed to be nonnull.
For all nonnegative integers $a$, $b$, we denote by $K_a$ a complete graph on $a$ vertices, and 
by $K_{a,b}$ a complete bipartite graph with parts of size $a$ and $b$.

Let $G$ be a graph.
The \defin{interior} of a path $P$ in $G$ with endpoints $x$ and $y$
is the set of vertices $V(P) \setminus \{x,y\}$.
For $X,Y \subset V(G)$, a path \defin{between} $X$ and $Y$ is a path in $G$ that is either 
a one-vertex path whose unique vertex is in $X \cap Y$, or 
a path with one endpoint in $X$ and the other endpoint in $Y$ such that 
the interior of $P$ is disjoint from $X \cup Y$. 

The \defin{neighborhood} of a vertex $u$ in a graph $G$, denoted by \defin{$N_G(u)$}, is the set $\{v \in V(G) \mid uv \in E(G)\}$.
For every set $X$ of vertices of a graph $G$, let $\defin{N_G(X)}=\bigcup_{u \in X} N_G(u) \setminus X$.

A \defin{rooted tree} is a tree $T$ with a distinguished vertex $r \in V(T)$ called the \defin{root} of $T$.
We denote by \defin{$\leaves{T}$} the set of all the \defin{leaves} of $T$,
i.e., $r$ when $V(T) = \{r\}$ or the vertices of $T$ distinct from $r$ and having degree $1$ in $T$.
The \defin{height} of $x\in V(T)$ is the number of vertices in the path between $x$ and $r$.
The \defin{vertex-height} of $T$ is the largest height of a vertex in $V(T)$.
A \defin{rooted forest} is a graph where each component is a rooted tree.

A \defin{tree decomposition}
of $G$ is a pair $\mathcal{W} = \big(T,(W_x \mid x \in V(T))\big)$,
where $T$ is a tree and $W_x \subseteq V(G)$ for every $x \in V(T)$, satisfying the following conditions:
\begin{enumerate}
    \item for every $u \in V(G)$, the set $\{x \in V(T) \mid u \in W_x\}$ induces a subtree of $T$ and
    \item for every edge $uv \in E(G)$, there exists $x \in V(T)$ such that $u,v \in W_x$.
\end{enumerate}
The sets $W_x$ are called the \defin{bags} of $\mathcal{W}$.
The \defin{width} of $\mathcal{W}$ is $\max_{x \in V(T)} |W_x|-1$,
and the \defin{treewidth} of $G$,
denoted by $\tw(G)$,
is the minimum width of a tree decomposition of $G$.

Let $G$ be a graph, let $S \subseteq V(G)$, and let $\mathcal{P}$ be a partition of $S$.
A \defin{tree decomposition} of $(G,\mathcal{P})$
is a pair $\big(T,(W_x \mid x \in V(T))\big)$ where $T$ is a tree and $W_x \subseteq \mathcal{P}$ for every $x \in V(T)$ such that
\begin{enumerate}
    \item for every $u \in S$, $\{x \in V(T) \mid u \in \bigcup W_x\}$ induces a subtree of $T$,
    \item for every edge $uv$ of $G[S]$, there exists $x \in V(T)$ such that $u,v \in \bigcup W_x$, and
    \item for every connected component $C$ of $G-S$, there exists $x \in V(T)$ such that $N_G(V(C)) \subseteq \bigcup W_x$.
\end{enumerate}
The \defin{width}\index{width} of this tree decomposition is then $\max_{x \in V(T)} |W_x| -1$,
and the \defin{treewidth} of $(G,\mathcal{P})$,
denoted by $\tw(G,\mathcal{P})$, is the minimum width of a
tree decomposition of $(G,\mathcal{P})$.

\begin{obs}[Observation~35~\cite{DJMMUW20}]\label{obs:partition_to_product_structure}
    Let $c$ and $k$ be positive integers.
    For every graph $G$, the following are equivalent:
    \begin{enumerate}[label={\normalfont(\arabic*)}]
        \item there exists a graph $H$ with $\tw(H) < k$ such that $G \subseteq H \boxtimes K_c$;
        \item there exists a partition $\mathcal{P}$ of $V(G)$ such that $\tw(G,\mathcal{P}) < k$, and $|P| \leq c$ for every $P \in \mathcal{P}$.
    \end{enumerate}
\end{obs}

The following lemma is folklore. See for example~\cite[Lemma~8]{DHHJLMMRW24}.

\begin{lemma}[Folklore]\label{lemma:helly_with_LCA}
    Let $t$ be a positive integer and let $G$ be a graph with $\tw(G)<t$.
    For every family $\mathcal{F}$ of connected subgraphs of $G$,
    for every positive integer $d$,
    either
    \begin{enumerate}
        \item there are $d+1$ pairwise vertex-disjoint members of $\mathcal{F}$, or
        \item there is a set $Z \subseteq V(G)$ of size at most $t(2d-1)$ such that
            \begin{enumerate}[label={\normalfont(\alph*)}]
                \item $V(F) \cap Z \neq \emptyset$ for every $F \in \mathcal{F}$, and
                \item for every component $C$ of $G-Z$, $|N_G(V(C))| \leq 2t$.
            \end{enumerate}
    \end{enumerate}
\end{lemma}

Let $G, X$ be two graphs.
A \defin{topological model} (also called subdivision) of $X$ in $G$ is a pair $\big((u_x \mid x \in V(X)), (P_{xy} \mid xy \in E(X))\big)$ where
\begin{enumerate}
    \item $(u_x \mid x\in V(X))$ are pairwise distinct vertices of $G$,
    \item $(P_{xy} \mid xy \in E(X))$ are pairwise internally disjoint paths in $G$,
    \item $u_x$ and $u_y$ are the endpoints of $P_{xy}$ for every $xy \in E(X)$, and
    \item $V(P_{xy}) \cap \{u_z \mid z \in V(X)\} = \{u_x, u_y\}$ for every $xy \in E(X)$.
\end{enumerate}
If there is a topological model of $X$ in $G$, then we say that $X$ is a \defin{topological minor} of $G$.
Otherwise, we say that $G$ is \defin{$X$-topological-minor-free}.

Let $F$ be a rooted forest.
The \defin{closure} of $F$, denoted by \defin{$\clos(F)$},
is the graph with vertex set $V(F)$ and edge set $E(F) \cup \{xy \mid \text{$x,y \in V(F)$ distinct, $y$ ancestor of $x$}\}$.
Let $G$ be a graph.

The \defin{treedepth} of $G$, denoted by \defin{$\td(G)$},
is the minimum integer $h$ such that there is a rooted forest $F$ on the vertex set $V(G)$ with vertex-height at most $h$ such that
$E(G) \subseteq E(\clos(F))$.

In the proof of \Cref{thm:main}, we build inductively subdivisions of a given graph via auxiliary objects, called ``$(X,\mathcal{F})$-structures'',
which we now define (see also \Cref{fig:TFstructure}).
Let $X$ be a rooted forest,
let $G$ be a graph,
and let $\mathcal{F}$ be a family of connected subgraphs of $G$.
For a subgraph $H$ of $G$, we denote by \defin{$\mathcal{F}\vert_H$} the family
of all the members $F$ of $\mathcal{F}$ which are subgraphs of $H$.
A \defin{$(X,\mathcal{F})$-structure} in $G$ is a
pair $\big((U_x \mid x \in V(X)), (P_{xy} \mid xy \in E(\clos(X)))\big)$ where
\begin{enumerate}
    \item $(U_x \mid x \in V(X))$ are pairwise disjoint subsets of $V(G)$ inducing connected subgraphs of $G$ such that for every $x \in V(X)$, if $x$ is a leaf in $X$, then there exists $F \in \mathcal{F}$ with $V(F) \subseteq U_x$, and otherwise $U_x$ is a singleton;
    \item $(P_{xy} \mid xy \in E(\clos(X)))$ are pairwise internally disjoint paths in $G$;
    \item $P_{xy}$ is a path in $G$ between $U_x$ and $U_y$ for every $xy \in E(\clos(X))$; and
    \item $V(P_{xy})\cap \bigcup_{z \in V(X)} U_z \subseteq U_x\cup U_y$ for every $xy \in E(\clos(X))$.
\end{enumerate}

\begin{figure}[tp]
    \centering
    \includegraphics{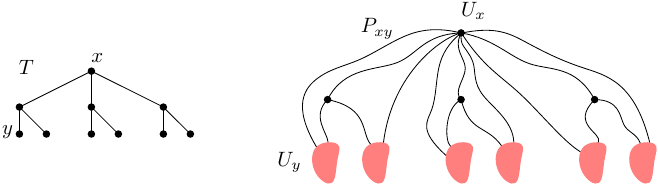}
    \caption{A $(T,\mathcal{F}$)-structure $\big((U_x \mid x \in V(T)), (P_{xy} \mid xy \in E(\clos(T)))\big)$ where $T$ is the rooted tree depicted on the left side. 
        Members of $\mathcal{F}$ are depicted in red.}
    \label{fig:TFstructure}
\end{figure}

\section{A colorful lemma}\label{sec:colorful}

The goal of this section is to prove a Ramsey-type lemma (\Cref{lemma:colorful_lemma}) which will be used in the proof of \Cref{thm:main}.
We start with the following well-known statement, whose proof is provided for completeness.

\begin{lemma}\label{lemma:boring_pigeon_hole_lemma}
    Let $k$, $d$, and $h$ be positive integers.
    There exists a positive integer $N(k,d,h)$ such that,
    for every mapping $f \colon [N(k,d,h)]^h \to [k]$,
    there exists $X_1, \dots, X_h \subseteq [N(k,d,h)]$, all of size $d$,
    such that $f$ is constant on $X_1 \times \dots \times X_h$.
\end{lemma}

\begin{proof}
    We proceed by induction on $h$.
    If $h=1$, then the result follows from the pigeonhole principle
    for $N(k,d,1) = k(d-1)+1$.
    Now suppose $h \geq 2$.
    Let
    \[\textstyle
        N(k,d,h) = k\binom{N(k,d,h-1)}{d}^{h-1}(d-1) + 1. 
    \]
    For every $x \in [N(k,d,h)]$, by the induction hypothesis,
    there exists $i(x) \in [k]$ and $X_2(x), \dots, X_h(x) \subseteq [N(k,d,h-1)]$ all of size $d$ 
    such that $f(\{x\} \times X_2(x) \times \dots \times X_h(x)) = \{i(x)\}$.
    By the pigeonhole principle,
    there exists $X_1 \subseteq [N(k,d,h)]$ of size $d$
    such that $x \mapsto (i(x), X_2(x), \dots, X_h(x))$ is constant on $X_1$.
    Then, for $(X_2, \dots, X_h) = (X_2(x), \dots, X_h(x))$ 
    for some arbitrary $x \in X_1$,
    we have that $f$ is constant on $X_1 \times \dots \times X_h$.
\end{proof}

\begin{lemma}\label{lemma:uncolor_the_leaves}
    Let $h$ and $k$ be positive integers with $h \geq 2$, 
    let $X$ be a rooted tree of vertex-height at most $h$,
    and let $L_X$ be the set of all the leaves of $X$ of height $h$.
    There exists a rooted tree $Y$ of vertex-height at most $h$ such that,
    if $L_Y$ is the set of all the leaves of $Y$ of height $h$,
    the following holds.
    For every $S_1, \dots, S_k \subseteq L_Y$ such that $\bigcup_{i \in [k]} S_i = L_Y$, 
    there exists $i \in [k]$ and an injective mapping $\phi \colon V(X) \to V(Y)$ such that
    \begin{enumerate}[label={\normalfont(\alph*)}]
        \item $\phi(x) \phi(x') \in E(Y)$ for every $xx' \in E(X)$,
        \item $\phi$ maps the root of $X$ to the root of $Y$, and
        \item $\phi(x) \in S_i$ for every $x \in L_X$.
    \end{enumerate}
\end{lemma}

\begin{proof}
    For every positive integer $d$, 
    let $T_{h,d}$ be the complete $d$-ary tree of vertex-height $h$
    that is the rooted tree defined by
    \[\textstyle
        V(T_{h,d}) = \bigcup_{i=0}^{h-1} [d]^i \qquad \text{and} \qquad 
        E(T_{h,d}) = \bigcup_{i=0}^{h-2} \{(x,xa) \mid x \in [d]^i, a \in [d]\},
    \]
    with the empty tuple $()$ being the root.
    Let $d = |V(X)|$.
    It follows that $X$ is a subtree of $T_{h,d}$ where the root is mapped to the root, and $L_X$ is mapped to $\leaves{T_{h,d}}$.
    Hence, it is enough to consider the case $X=T_{h,d}$.
    Then,
    let $Y = T_{h,N(k,d,h-1)}$ where $N(k,d,h-1)$ is the positive integer given by \Cref{lemma:boring_pigeon_hole_lemma}.
    Then, for every $S_1, \dots, S_k \subseteq \leaves{Y}$ with $\bigcup_{i \in [k]} S_i = \leaves{Y}$,
    consider $f \colon \leaves{Y} \to [k]$ such that $u \in S_{f(u)}$ for every $u \in \leaves{Y}$.
    Since $\leaves{Y} = [N(k,d,h-1)]^{h-1}$, and by \Cref{lemma:boring_pigeon_hole_lemma},
    there exists $X_1, \dots, X_{h-1} \subseteq [N(k,d,h-1)]$ all of size $d$
    such that $f$ is constant on $X_1 \times \dots \times X_{h-1}$. 
    In other words, there exists $i \in [k]$ such that $X_1 \times \dots \times X_{h-1} \subseteq S_i$.
    For every $j \in [h-1]$, fix a bijection $\phi_j \colon [d] \to X_j$.
    Then, for every $j \in \{0, \dots, h-1\}$, for every $x = (x_1, \dots, x_j) \in [d]^j$,
    let
    \[
        \phi(x) = (\phi_1(x_1), \dots, \phi_j(x_j)).
    \]
    Then, $\phi \colon V(X) \to V(Y)$ is as desired.
\end{proof}

\begin{lemma}\label{lemma:colorful_lemma}
    Let $h$ and $k$ be positive integers with $h \geq 2$
    and let $X$ be a rooted tree of vertex-height at most $h$.
    There exists a rooted tree $Y$ of vertex-height at most $h$ such that
    the following holds.
    For every graph $G$,
    for all families $\mathcal{F}_1, \dots, \mathcal{F}_k$ of connected subgraphs of $G$,
    if there is a $\big(Y, \bigcup_{i \in [k]} \mathcal{F}_i\big)$-structure in $G$,
    then there exists $i \in [k]$ such that there is an $(X,\mathcal{F}_i)$-structure in $G$.
\end{lemma}

\begin{proof}
    If $X$ has only one vertex, then taking $Y = X$ suffices.
    We assume that $X$ has more vertices.
    Let $X'$ be obtained from $X$ by subdividing $h-h'$ times for every leaf $\ell$ the edge incident to $\ell$, where $h'$ is the height of $\ell$ in $X$.
    Note that $X'$ has vertex-height $h$,
    every leaf of $X'$ has height $h$,
    and every $(X',\mathcal{F})$-structure contains an $(X, \mathcal{F})$-structure
    for every family $\mathcal{F}$ of connected subgraphs of $G$.

    Let $Y$ be a rooted tree given by
    \Cref{lemma:uncolor_the_leaves} for $h$, $k$, and $X'$.
    Let $G$ be a graph,
    let $\mathcal{F}_1, \dots, \mathcal{F}_k$ be families of connected subgraphs of $G$,
    and let $\big((U_y \mid y \in V(Y)), (P_{yz} \mid yz \in E(\clos(Y)))\big)$
    be a $\big(Y, \bigcup_{i \in [k]} \mathcal{F}_i\big)$-structure in $G$.
    For every $i \in [k]$, let $S_i$ be the set of all the leaves $y$ of $Y$ such that $G[U_y]$ contains a member of $\mathcal{F}_i$.
    By the definition of $Y$,
    there exists $i \in [k]$ and an injective mapping $\phi \colon V(X') \to V(Y)$ such that, 
    \begin{enumerate}[label={\normalfont\ref{lemma:uncolor_the_leaves}.(\alph*)}]
        \item $\phi(x) \phi(x') \in E(Y)$ for every $xx' \in E(X')$, 
        \item $\phi$ maps the root of $X'$ to the root of $Y$, and 
        \item $\phi(x) \in S_i$ for every $x \in \leaves{X'}$.
    \end{enumerate}
    In particular, $xy \in E(\clos(X'))$ implies $\varphi(x)\varphi(y) \in E(\clos(Y))$.
    For every $x \in V(X')$, let $U'_x = U_{\phi(x)}$,
    and for every $xy \in E(\clos(X'))$, let $P'_{xy} = P_{\phi(x)\phi(y)}$.
    It follows that $\big((U'_x \mid x \in V(X')), (P'_{xy} \mid xy \in E(\clos(X')))\big)$
    is an $(X', \mathcal{F}_i)$-structure in $G$.
    Therefore, there is an $(X,\mathcal{F}_i)$-structure in $G$.
\end{proof}

\section{Proof of Theorem~\ref{thm:main}}\label{sec:main_proof}

In this section, we prove \Cref{thm:main},
which will be a consequence of the following technical statement.

\begin{thm}\label{thm:main_induction}
    Let $h$ and $t$ be positive integers with $h \geq 2$,
    and let $X$ be a rooted tree of vertex-height at most $h$.
    There exists a positive integer $c_{\ref{thm:main_induction}}(h,X,t)$ such that the following holds.
    Let $G$ be a 
    graph of treewidth less than $t$, let $R\subseteq V(G)$ with 
    $1 \leq |R| \leq (2^{h-1}-1) \cdot 2t$, 
    and let $\mathcal{F}$ be a family of connected subgraphs of $G-R$
    such that there is no $(X,\mathcal{F})$-structure in $G$.
    There exist $S \subseteq V(G)$,
    a partition $\mathcal{P}$ of $S$,
    and a tree decomposition $\big(T, (W_x \mid x \in V(T))\big)$ of $(G, \mathcal{P})$ such that
    \begin{enumerate}[label={\normalfont(\Alph*)}]
        \item $V(F) \cap S \neq \emptyset$ for every $F \in \mathcal{F}$, \label{item:thm:main_induction:hitting}
        \item $R \in \mathcal{P}$, \label{item:thm:main_induction:R_in_P}
        \item $|P| \leq c_{\ref{thm:main_induction}}(h,X,t)$ for every $P \in \mathcal{P} \setminus \{R\}$, \label{item:thm:main_induction:width_of_P}
        \item $|W_x| \leq 2^h-2$ for every $x \in V(T)$, \label{item:thm:main_induction:bags_are_small}
        \item $|W_x| \leq 2^{h-1}$ for every $x \in V(T)$ with $R \in W_x$, \label{item:thm:main_induction:bags_with_R_are_very_small}
        \item for every component $C$ of $G-S$, $|N_G(V(C))| \leq (2^{h}-2) \cdot 2t$. \label{item:thm:main_induction:interface}
    \end{enumerate}
\end{thm}

\begin{proof}
    Let $s$ be the root of $X$,
    let $d'$ be the degree of $s$ in $X$,
    let $d = \max\{d',2\}$,
    and let $X_0$ be obtained from $X$ by contracting the set $N_X(s) \cup \{s\}$
    into a single vertex $s_0$.
    Then, let $X_1$ be obtained from $X_0$ by adding for each $x \in V(X_0) \setminus \leaves{X_0}$ a new leaf $\ell(x)$ 
    adjacent to $x$.
    See \Cref{fig:claim:no_F'r_structure}.
    Note that $X_1$ has vertex-height at most $h-1$.
    Let 
    \[
        k = (2^h-1) \cdot 2t.
    \]
    If $h=2$, let $Y_1$ be the single-vertex rooted tree.
    Otherwise, by \Cref{lemma:colorful_lemma}, 
    there exists a rooted tree $Y_1$ of vertex-height at most $h-1$ such that 
    for every graph $G'$,
    for all families $\mathcal{F}_1, \dots, \mathcal{F}_k$ of connected subgraphs of $G'$,
    if there is a $(Y_1, \bigcup_{i \in [k]} \mathcal{F}_i)$-structure in $G'$,
    then there exists $i \in [k]$ such that $G'$ contains an $(X_1,\mathcal{F}_i)$-structure.
    Let 
    \[
        c_{\ref{thm:main_induction}}(h,X,t) = 
        \begin{cases}
            1 & \textrm{if $|V(X)| = 1$,}\\
            t(2k(d-1)-1) & \textrm{if $|V(X)|\geq 2$ and $h=2$,} \\
            \max\{c_{\ref{thm:main_induction}}(h-1,Y_1,t), t(2k(d-1)-1)\} & \textrm{if $|V(X)|\geq 2$ and $h>2$.}
        \end{cases}
    \]

    We show by induction on $(h, |V(G-R)|)$ lexicographically that $c_{\ref{thm:main_induction}}$ witnesses the statement. 
    We call a tuple $(h,t,X,G,R,\mathcal{F})$ as in the statement an instance.

    First suppose that $X$ has only one vertex. 
    Since $G$ has no $(X, \mathcal{F})$-structure,
    $\mathcal{F} = \emptyset$.
    Therefore, 
    the statement holds for $S=R$, $\mathcal{P} = \{R\}$, and a trivial one-bag tree decomposition of $(G,\mathcal{P})$.
    From now on, we suppose that $X$ has at least two vertices.

    Similarly, if $V(G)=R$, then set $S=R$ and $\mathcal{P} = \{R\}$ and the statement holds. 
    From now on, we suppose that $V(G) \setminus R \neq \emptyset$.

    For every $r \in R$,
    let $\mathcal{F}'_r$ be the family of all the connected subgraphs $H$ 
    of $G-R$
    such that
    \begin{enumerate}
        \item there exists $F \in \mathcal{F}$ such that $F \subseteq H$, and
        \item $N_G(r) \cap V(H) \neq \emptyset$.
    \end{enumerate}

    \begin{claim}\label{claim:no_F4_rich_Y2}
        There are no $k(d-1)+1$ pairwise disjoint $(Y_1,\bigcup_{r \in R} \mathcal{F}'_r)$-structures in $G-R$.
    \end{claim}
    
    \begin{proofclaim}
        Suppose to the contrary that there are $k(d-1)+1$ pairwise disjoint $(Y_1,\bigcup_{r \in R} \mathcal{F}'_r)$-structures in $G-R$.
        For each copy $Y'$ of $Y_1$ in a structure, we proceed as follows.
        By the definition of $Y_1$ applied to the $(Y', \bigcup_{r \in R} \mathcal{F}'_r\vert_{G_{Y'}})$-structure $\big((U'_x \mid x \in V(Y')), (P'_{xy} \mid xy \in E(\clos(Y')))\big)$ 
        in 
        \[\textstyle G_{Y'} = G\left[\bigcup_{x \in V(Y')} U'_x \cup \bigcup_{xy \in E(\clos(Y'))} V(P'_{xy})\right],\]
        there exists $r(Y') \in R$ such that $G_{Y'}$ contains an $(X_1, \mathcal{F}'_{r(Y')})$-structure $\big((U_{Y',x} \mid x \in V(X_1)), (P_{Y',xy} \mid xy \in E(\clos(X_1)))\big)$ by \Cref{lemma:colorful_lemma}.
        For every $x \in V(X_1)\setminus \leaves{X_1}$, 
        $G[U_{Y',\ell(x)}]$ contains a member of $\mathcal{F}'_{r(Y')}$, 
        and so $U_{Y',\ell(x)} \cup V(P_{Y',x \ell(x)}) \cup \{r(Y')\}$ contains a path between $r(Y')$ and $U_{Y',x}$ in $G[V(G_{Y'}) \cup \{r(Y')\}]$.
        Let $P_{Y',xs}$ be such a path.
        For every $x \in \leaves{X_0}$, 
        let $P_{Y',xs}$ be a one-edge path between $r(Y')$ and $U_{Y',x}$ in $G[V(G_{Y'}) \cup \{r(Y')\}]$, which exists since $G[U_{Y',x}]$ contains a member of $\mathcal{F}'_{r(Y')}$.
        See \Cref{fig:claim:no_F'r_structure}.

        By the pigeonhole principle, there exist $d$ distinct copies $Y'_1, \dots, Y'_d$ of $Y_1$ giving $(Y_i',\bigcup_{r \in R} \mathcal{F}'_r)$-structures in $G-R$ for each $i \in [d]$ such that
        $r(Y'_1) = \dots = r(Y'_d)$.
        Let $r$ be this element of $R$.
        By construction, the union of $d$ disjoint copies of $X_0$ contains $X-s$.
        Therefore, $G-R$ contains an $(X-s, \mathcal{F}'_r)$-structure $\big((U_x \mid x \in V(X-s)), (P_{xy} \mid xy \in E(\clos(X-s)))\big)$
        which, together with $U_s = \{r\}$ and the previously constructed paths $P_{Y'_i,xs}$ for $i \in [d]$ and $x \in V(X_0)$,
        yields an $(X,\mathcal{F}'_r)$-structure and so an $(X,\mathcal{F})$-structure in $G$ by definition of $\mathcal{F}'_r$,
        a contradiction.
        \begin{figure}[tp]
            \centering
            \includegraphics{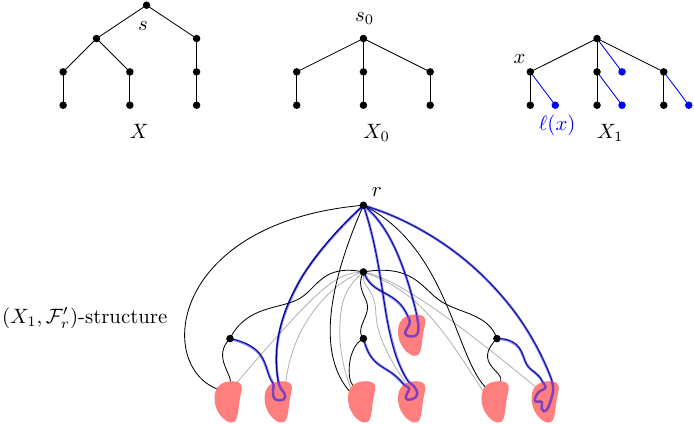}
            \caption{Illustration for the proof of \Cref{claim:no_F4_rich_Y2}.
                Using the new leaves $\ell(x)$ for $x \in V(X_0) \setminus \leaves{X_0}$, we create the blue paths. Members of $\mathcal{F}'_r$ are depicted in red.}
            \label{fig:claim:no_F'r_structure}
        \end{figure}
    \end{proofclaim}

    Let $\mathcal{F}''$ be the family of all the connected subgraphs $H$ of $G-R$ such that
    $H$ contains a $\big(Y_1,\bigcup_{r \in R} \mathcal{F}'_r \vert_H\big)$-structure.
    By \Cref{claim:no_F4_rich_Y2}, \Cref{lemma:helly_with_LCA} implies that there is a set $Z \subseteq V(G-R)$ of size at most $t(2k(d-1)-1)$ such that
    \begin{enumerate}[label={\normalfont(Z\alph*)}]
        \item $V(F) \cap Z \neq \emptyset$ for every $F \in \mathcal{F}''$, and \label{item:def_Z:hitting}
        \item $|N_{G-R}(V(C))| \leq 2t$ for every component $C$ of $G - (R \cup Z)$. \label{item:def_Z:interface}
    \end{enumerate}
    In the case that $Z$ is empty, 
    we add to $Z$ an arbitrary vertex of $V(G-R)$. (Recall that $V(G-R) \neq \emptyset$.)  
    Note that this retains both properties: \ref{item:def_Z:hitting} and \ref{item:def_Z:interface}. 
    Thus, we assume that $Z$ is nonempty.
    Also, if $V(G) = R \cup Z$, then the result is clear for
    $S = R \cup Z$ and $\mathcal{P} = \{R,Z\}$.
    Now suppose $V(G - (R \cup Z)) \neq \emptyset$.

    \begin{claim}
        There exist $S_0 \subseteq V(G-(R \cup Z))$,
        a partition $\mathcal{P}_0$ of $S_0$,
        and a tree decomposition
        $\big(T_0, (W_{0,x} \mid x \in V(T_0))\big)$ of $(G-(R \cup Z), \mathcal{P}_0)$ such that
        \begin{enumerate}[label={\normalfont(S0\alph*)}]
            \item $V(F) \cap S_{0} \neq \emptyset$ for every $F \in \bigcup_{r \in R}\mathcal{F}'_r \vert_{G-(R \cup Z)}$, \label{item:prop_of_S0:hitting}
            \item $|P| \leq c_{\ref{thm:main_induction}}(h,X,t)$ for every $P \in \mathcal{P}_{0}$, \label{item:prop_of_S0:width_of_P0}
            \item $|W_{0,x}| \leq 2^{h-1}-2$ for every $x \in V(T_{0})$, and\label{item:prop_of_S0:width_of_W0}
            \item for every component $C'$ of $G-(R \cup Z \cup S_{0})$, $|N_{G - (R \cup Z)}(V(C'))| \leq (2^{h-1}-2) \cdot 2t$. \label{item:prop_of_S0:interface}
        \end{enumerate}
    \end{claim}

    \begin{proofclaim}
        If $h=2$,
        let $S_0 = \emptyset$ and $\mathcal{P}_0 = \emptyset$,
        let $T_0$ be a one-vertex rooted tree,
        and let $W_{0,x} = \emptyset$ for the unique vertex $x$ of $T_0$.
        Since in this case, $Y_1$ has only one vertex,
        and because $Z$ intersects every member of $\mathcal{F}''$,
        there is no member of $\bigcup_{r \in R} \mathcal{F}'_r$  contained in $G-(R \cup Z)$.
        Therefore,
        \ref{item:prop_of_S0:hitting}--\ref{item:prop_of_S0:interface} are satisfied.
        Now suppose $h \geq 3$.

        Let $R_0$ be an arbitrary singleton in $V(G - (R \cup Z))$.
        By \ref{item:def_Z:hitting}, there is no $(Y_1, \bigcup_{r \in R} \mathcal{F}'_r\vert_{G-(R \cup Z \cup R_0)})$-structure in $G - (R \cup Z)$.
        Hence, we apply the induction hypothesis to the instance
        $(h-1, t, Y_1, G-(R \cup Z), R_0, \bigcup_{r \in R}\mathcal{F}'_r\vert_{G-(R \cup Z \cup R_0)})$,
        there exist $S_{0} \subseteq V(G-(R \cup Z))$ with $R_0 \subseteq S_0$,
        a partition $\mathcal{P}_{0}$ of $S_{0}$,
        and a tree decomposition $\big(T_{0}, (W_{0,x} \mid x \in V(T_{0}))\big)$ of $(G-(R \cup Z), \mathcal{P}_{0})$ such that, in particular,
        \begin{enumerate}[label={\normalfont(\Alph*$'$)}]       
            \item $V(F) \cap S_{0} \neq \emptyset$ for every $F \in \bigcup_{r \in R}\mathcal{F}'_r\vert_{G-(R \cup Z \cup R_{0})}$, and since $R_0 \subseteq S_0$, \ref{item:prop_of_S0:hitting} holds; \label{item:S0C:hitting} 
            \addtocounter{enumi}{1}
            \item $|P| \leq c_{\ref{thm:main_induction}}(h-1,Y_1,t) \leq c_{\ref{thm:main_induction}}(h,X,t)$ for every $P \in \mathcal{P}_{0} \setminus \{R_{0}\}$, and since $|R_0| = 1 \leq  c_{\ref{thm:main_induction}}(h,X,t)$, \ref{item:prop_of_S0:width_of_P0} holds; \label{item:S0C:width_of_P0C}
            \item $|W_{0,x}| \leq 2^{h-1}-2$ for every $x \in V(T_{0})$, so \ref{item:prop_of_S0:width_of_W0} holds; \label{item:prop_of_S0:width_of_W0:S0C}
            \addtocounter{enumi}{1}
            \item for every component $C'$ of $G-(R \cup Z\cup S_0)$, $|N_{G-(R \cup Z)}(V(C'))| \leq (2^{h-1}-2) \cdot 2t$, so \ref{item:prop_of_S0:interface} holds. \label{item:prop_of_S0:interface:S0C}
        \end{enumerate}
        In particular, this proves the claim.
    \end{proofclaim}
    We now fix $S_0$, $\mathcal{P}_0$, and $\big(T_0, (W_{0,x} \mid x \in V(T_0))\big)$ as in the claim.

    Let $\mathcal{C}_1$ be the family of all the components $C$ of $G-(R\cup Z \cup S_0)$ with $\mathcal{F}\vert_C \neq \emptyset$.
    Let $C \in \mathcal{C}_1$.
    Since $C$ is disjoint from $S_0$ and $\mathcal{F}\vert_C \neq \emptyset$,
    by \ref{item:prop_of_S0:hitting},
    $N_G(V(C)) \cap R = \emptyset$.
    Let $R_{1,C}$ be the set $N_G(V(C))$ if $N_G(V(C)) \neq \emptyset$,
    and let $R_{1,C}$ be an arbitrary singleton in $Z$ if $N_G(V(C)) = \emptyset$.
    Then, let $G_C = G[V(C) \cup R_{1,C}]$.
    
    Since $Z \neq \emptyset$, we have $|V(G_C-R_{1,C})| = |V(C)| < |V(G-R)|$.
    Moreover, 
    \begin{align*}
        |R_{1,C}| &\leq \max\{1,|N_{G}(V(C)) \cap Z| + |N_{G}(V(C)) \cap S_0|\} \\ 
        &\leq \max\{1,|N_{G-R}(V(D))| + |N_{G-(R\cup Z)}(V(C))|\}\\
    \intertext{ where $D$ is a component of $G-(R\cup Z)$ containing $C$}
        &\leq 2t + (2^{h-1}-2) \cdot 2t = (2^{h-1}-1) \cdot 2t \textrm{ by \ref{item:def_Z:interface} and \ref{item:prop_of_S0:interface}.}
    \end{align*}
    Hence, by the induction hypothesis applied to the instance $(h,t,X,G_C,R_{1,C},\mathcal{F}\vert_C)$,
    there exist $S_{1,C} \subseteq V(G_C)$,
    a partition $\mathcal{P}_{1,C}$ of $S_{1,C}$,
    and a tree decomposition $\big(T_{1,C}, (W_{1,x} \mid x \in V(T_{1,C}))\big)$ of $(G_C, \mathcal{P}_{1,C})$ such that
    \begin{enumerate}[label={\normalfont(\Alph*$''$)}]
        \item $V(F) \cap S_{1,C} \neq \emptyset$ for every $F \in \mathcal{F}\vert_C$, \label{item:thm:main_induction:hitting:S1C}
        \item $R_{1,C} \in \mathcal{P}_{1,C}$,
        \item $|P| \leq c_{\ref{thm:main_induction}}(h,X,t)$ for every $P \in \mathcal{P}_{1,C} \setminus \{R_{1,C}\}$, \label{item:thm:main_induction:width_of_P1C:S1C}
        \item $|W_{1,x}| \leq 2^h-2$ for every $x \in V(T_{1,C})$, \label{item:prop_of_S1C:bags_small}
        \item $|W_{1,x}| \leq 2^{h-1}$ for every $x \in V(T_{1,C})$ with $R_{1,C} \in W_{1,x}$, \label{item:prop_of_S1C:bags_very_small}
        \item for every component $C'$ of $G_C-S_{1,C}$, $|N_{G_C}(V(C'))| \leq (2^{h}-2) \cdot 2t$. \label{item:prop_of_S1C:interface}
    \end{enumerate}

    Since $\big(T_0, (W_{0,x} \mid x \in V(T_0))\big)$ is a tree decomposition of $(G - (R \cup Z), \mathcal{P}_0)$,
    there exists $z_C \in V(T_0)$ such that $N_G(V(C)) \subseteq Z \cup \bigcup W_{0,z_C}$. Fix such a vertex $z_C$. 
    In both cases of the definition of $R_{1,C}$, we have $R_{1,C}\subseteq Z \cup \bigcup W_{0,z_C}$. 
    Since $R_{1,C} \in \mathcal{P}_{1,C}$, there exists $y_C \in V(T_{1,C})$ such that $R_{1,C} \in W_{1,y_C}$.
    Fix such a vertex $y_C$.

    Let
    \begin{align*}
          S &= R \cup Z \cup S_0 \cup \bigcup_{C \in \mathcal{C}_1} (S_{1,C}-R_{1,C}), \\
        \mathcal{P} &= \{R,Z\} \cup \mathcal{P}_0 \cup \bigcup_{C \in \mathcal{C}_1} \big(\mathcal{P}_{1,C} \setminus \{R_{1,C}\}\big).
    \end{align*}

    Suppose that the trees $T_{1,C}$ for $C \in \mathcal{C}_1$ have pairwise disjoint vertex sets disjoint from $V(T_0)$.
    Now, let $T$ be the tree defined by
    \begin{align*}
        V(T) &= V(T_0) \cup \bigcup_{C \in \mathcal{C}_1} V(T_{1,C}), \\
        E(T) &= E(T_0) \cup \bigcup_{C \in \mathcal{C}_1} \big(E(T_{1,C}) \cup \{y_C z_C\}\big).
    \end{align*}
    Then, for every $x \in V(T)$, let
    \[
        W_x = 
        \begin{cases}
            W_{0,x} \cup \{R,Z\} & \textrm{if $x \in V(T_0)$,} \\
            (W_{1,x} \setminus \{R_{1,C}\}) \cup W_{0,z_C} \cup \{Z\} & \textrm{if $x \in V(T_{1,C})$ for some $C \in \mathcal{C}_1$ and $R_{1,C} \in W_{1,x}$,} \\
            W_{1,x} & \textrm{if $x \in V(T_{1,C})$ for some $C \in \mathcal{C}_1$ and $R_{1,C} \not\in W_{1,x}$.} \\
        \end{cases}
    \]
    We claim that $S$, $\mathcal{P}$, and $\big(T,(W_x \mid x \in V(T))\big)$ are as wanted.

    First, $\mathcal{P}$ is indeed a partition of $S$.
    We now show that $\big(T,(W_x \mid x \in V(T))\big)$ is a tree decomposition of $(G,\mathcal{P})$.
    Let $P \in \mathcal{P}$. We show that $\{x \in V(T) \mid P \in W_x\}$ induces a subtree of $T$.
    If $P = R$, then $\{x \in V(T) \mid P \in W_x\} = V(T_0)$, which induces $T_0$ in $T$.
    If $P \in \mathcal{P}_0$, then 
    \[
        \{x \in V(T) \mid P \in W_x\} =
            \bigcup_{x \in V(T_0), P \in W_{0,x}} \left(\{x\} \cup \bigcup_{C \in \mathcal{C}_1, z_C=x} \{y \in V(T_{1,C}) \mid R_{1,C} \in W_{1,y}\}\right).
    \]
    
    Now, $\{x \in V(T_0) \mid P \in W_{0,x}\}$ induces a subtree of $T_0$ 
    since $\big(T_0, (W_{0,x} \mid x \in V(T_0))\big)$ is a tree decomposition of $(G - (R \cup Z),\mathcal{P}_0)$.
    Second, for every $x \in \{x \in V(T_0) \mid P \in W_{0,x}\}$ and for every $C \in \mathcal{C}_1$ with $z_C=x$, 
    $\{y \in V(T_{1,C}) \mid R_{1,C} \in W_{1,y}\}$ induces a subtree of $T_{1,C}$ which contains 
    $y_C$, and $z_C y_C \in E(T)$.
    Altogether, this implies that $\{x \in V(T) \mid P \in W_x\}$ induces a subtree of $T$.
    Similarly,
    if $P=Z$, then
    \[
        \{x \in V(T) \mid P \in W_x\} =
            \bigcup_{x \in V(T_0)} \left(\{x\} \cup \bigcup_{C \in \mathcal{C}_1, z_C=x} \{y \in V(T_{1,C}) \mid R_{1,C} \in W_{1,y}\}\right),
    \]
    which induces a subtree of $T$ for the same reason.
    If $P \in \mathcal{P}_{1,C}$ for some $C \in \mathcal{C}_1$, then
    $\{x \in V(T) \mid P \in W_x\} = \{x \in V(T_{1,C}) \mid P \in W_{1,x}\}$, which is a connected subtree of $T_{1,C}$, and so of $T$,
    since $\big(T_{1,C}, (W_{1,x} \mid x \in V(T_{1,C}))\big)$ is a tree decomposition of $(G_C,\mathcal{P}_{1,C})$.
    This shows that $\{x \in V(T) \mid P \in W_x\}$ induces a subtree of $T$ for every $P \in \mathcal{P}$.

    Consider an edge $uv \in E(G[S])$.
    We want to show that $u,v \in \bigcup W_x$ for some $x \in V(T)$.
    If $u \in R$ then $v \in R \cup Z \cup S_0$ since there is no edge between $V(C)$ and $R$ for every $C \in \mathcal{C}_1$.
    Hence, either $v \in R\cup Z$ and so $u,v \in \bigcup W_x$ for any $x \in V(T_0)$,
    or $v \in S_0$. 
    In this latter case, since $\big(T_0, (W_{0,x} \mid x \in V(T_0))\big)$ is a tree decomposition of $(G-(R \cup Z), \mathcal{P}_0)$,
    there exists $x \in V(T_0)$ such that $v \in \bigcup W_{0,x}$, and so $u,v \in R \cup Z \cup \bigcup W_{0,x} = \bigcup W_x$.
    We conclude symmetrically if $v \in R$.
    Now suppose $u,v \not\in R$.
    If $u,v \in S_0 \cup Z$, then, since $\big(T_0, (W_{0,x} \mid x \in V(T_0))\big)$ is a tree decomposition of $(G-(R \cup Z), \mathcal{P}_0)$,
    there exists $x \in V(T_0)$ such that $u,v \in Z \cup \bigcup W_{0,x} \subseteq \bigcup W_x$.
    Now suppose without loss of generality that $v \not\in S_0 \cup Z$,
    and so $v \in S_{1,C}-R_{1,C}$ for some $C \in \mathcal{C}_1$.
    This implies $u \in S_{1,C}$.
    Since $\big(T_{1,C}, (W_{1,x} \mid x \in V(T_{1,C}))\big)$ is a tree decomposition of $(G_C, \mathcal{P}_{1,C})$,
    and because $uv$ is an edge in $G_C$,
    there exists $x \in V(T_{1,C})$ such that $u,v \in \bigcup W_{1,x} \subseteq \bigcup W_x$.
    This proves that for every $uv \in E(G[S])$,
    there exists $x \in V(T)$ such that $u,v \in \bigcup W_x$.

    We now show that for every component $C'$ of $G-S$,
    there exists $x \in V(T)$ such that $N_G(V(C')) \subseteq \bigcup W_x$, and $|N_G(V(C'))| \leq (2^h-2) \cdot 2t$.
    Let $C'$ be a component of $G-S$, and let $C$ be the component of $G - (R \cup Z \cup S_0)$ containing $C'$.
    First suppose that $C \not \in \mathcal{C}_1$, and so $C'=C$.
    Since $\big(T_0, (W_{0,x} \mid x \in V(T_0))\big)$ is a tree decomposition of $(G-(R \cup Z), \mathcal{P}_0)$,
    there exists $x \in V(T_0)$ such that $N_{G - (R \cup Z)}(V(C')) \subseteq \bigcup W_{0,x}$.
    Since $\bigcup W_x = \bigcup W_{0,x} \cup R \cup Z$, we deduce that $N_{G}(V(C')) \subseteq \bigcup W_{x}$.
    Moreover, $|N_{G - (R \cup Z)}(V(C'))| \leq (2^{h-1}-2) \cdot 2t$ by \ref{item:prop_of_S0:interface},
    $|N_{G-R}(V(C')) \cap Z| \leq |N_{G-R}(V(D)) \cap Z| \leq 2t$ by \ref{item:def_Z:interface} where $D$ is a component of $G-(R\cup Z)$ containing $C'$,
    and $|R| \leq (2^{h-1}-1) \cdot 2t$ by assumption.
    Altogether, this implies 
    \begin{align*}
        |N_G(V(C'))| &\leq |N_G(V(C'))\cap R| + |N_G(V(C'))\cap Z| + |N_G(V(C'))\cap S_0| \\ 
        &\leq |R| + |N_{G-R}(V(C'))\cap Z| + |N_{G-(R\cup Z)}(V(C'))| \\ 
        &\leq (2^{h-1}-1) \cdot 2t + 2t + (2^{h-1}-2) \cdot 2t = (2^h-2) \cdot 2t.
    \end{align*}
    This concludes the case $C \not\in \mathcal{C}_1$.
    Now suppose $C \in \mathcal{C}_1$.
    In particular, $N_G(V(C)) \cap R = \emptyset$, and $C'$ is a component of $G_C-S_{1,C}$.
    Hence, by \ref{item:prop_of_S1C:interface}, $N_{G_C}(V(C')) = N_G(V(C'))$ has size at most
    $(2^{h}-2) \cdot 2t$.
    Moreover, 
    there is $x \in V(T_{1,C})$ such that $N_{G_C}(V(C')) \subseteq V(G_C) \cap \bigcup W_{1,x} \subseteq \bigcup W_x$,
    and so $N_G(V(C')) \subseteq \bigcup W_x$.
    This proves that for every component $C'$ of $G-S$,
    there exists $x \in V(T)$ such that $N_G(V(C')) \subseteq \bigcup W_x$, and $|N_G(V(C'))| \leq (2^h-2) \cdot 2t$.
    The latter implies that \ref{item:thm:main_induction:interface} holds.

    This concludes the proof that $\big(T,(W_x \mid x \in V(T))\big)$ is a tree decomposition of $(G,\mathcal{P})$,
    and that \ref{item:thm:main_induction:interface} holds.
    We now show \ref{item:thm:main_induction:hitting}--\ref{item:thm:main_induction:bags_with_R_are_very_small}.
    
    Let $F \in \mathcal{F}$. We show that $V(F) \cap S \neq \emptyset$.
    If $V(F) \cap (R\cup Z \cup S_0) \neq \emptyset$ then we are done.
    Otherwise, $F \in \mathcal{F} \vert_C$ for some $C \in \mathcal{C}_1$,
    and so $V(F) \cap S_{1,C} \neq \emptyset$ by \ref{item:thm:main_induction:hitting:S1C}.
    Since $V(F) \cap R_{1,C} = \emptyset$,
    we obtain $V(F) \cap (S_{1,C} \setminus R_{1,C}) \neq \emptyset$,
    and so $V(F) \cap S \neq \emptyset$.
    This proves \ref{item:thm:main_induction:hitting}.

    Item \ref{item:thm:main_induction:R_in_P} holds by construction,
    and \ref{item:thm:main_induction:width_of_P} follows from \ref{item:prop_of_S0:width_of_P0},
    \ref{item:thm:main_induction:width_of_P1C:S1C}, and the fact that $|Z| \leq t(2k(d-1)-1) \leq c_{\ref{thm:main_induction}}(h,X,t)$.
    We now show \ref{item:thm:main_induction:bags_are_small} and \ref{item:thm:main_induction:bags_with_R_are_very_small}.
    Let $x \in V(T)$.
    If $x \in V(T_0)$, then $|W_x| = 2 + |W_{0,x}| \leq 2^{h-1}$ by \ref{item:prop_of_S0:width_of_W0}.
    Otherwise, $x \in V(T_{1,C})$ for some $C \in \mathcal{C}_1$, and $R \not\in W_x$.
    If $R_{1,C} \in W_{1,x}$, then $|W_{1,x}| \leq 2^{h-1}$ by \ref{item:prop_of_S1C:bags_very_small},
    and so by \ref{item:prop_of_S0:width_of_W0},
    we deduce $|W_x| = |W_{1,x}|-1 + |W_{0,z_C}| + 1 \leq 2^{h-1}-1 + (2^{h-1}-2) + 1 = 2^h-2$.
    Finally, if $R_{1,C} \not\in W_{1,x}$, then $|W_x| = |W_{1,x}| \leq 2^h-2$ by \ref{item:prop_of_S1C:bags_small}.
    This proves \ref{item:thm:main_induction:bags_are_small} and \ref{item:thm:main_induction:bags_with_R_are_very_small},
    and concludes the proof of the theorem.
\end{proof}

We now apply \Cref{thm:main_induction} to show \Cref{thm:main}.

\begin{proof}[Proof of \Cref{thm:main}]
    We may assume that $X$ is nonnull, since otherwise the assertion is vacuously true.
    If $\td(X)=1$, then $X$ is edgeless, and so every $X$-topological-minor-free graph has at most $|V(X)|-1$ vertices.
    In this case, the result is clear for $c(X,t)=\max\{|V(X)|-1,1\}$. 
    Thus, we may assume that $\td(X) \geq 2$ and in particular, $h \geq 2$.
    Let $h' = \td(X)$ and note that $h' \leq h$. 
    
    If $Z$ is a subgraph of $Z'$ and the number $c(Z',t)$ witnesses the theorem for $Z'$ and $t$, then setting $c(Z,t) = c(Z',t)$ gives a number witnessing the theorem for $Z$ and $t$; here we use the fact that $\td(Z) \leq \td(Z')$.
    For a large enough integer $\Delta$, each component of $X$ is isomorphic to a subgraph of the closure of the complete $\Delta$-ary tree of vertex-height $h'$.
    Let $Y_0$ be such a tree and let $X' = \clos(Y_0)$.
    Then, every component of $X$ is a subgraph of $X'$ and $\td(X') = h'$.
    Let $d$ be the number of components of $X$ and let $X^+$ be the disjoint union of $d$ copies of $X'$.
    Note that $X$ is a subgraph of $X^+$ and $\td(X^+) = h'$.
    For every $y \in \leaves{Y_0}$, add a leaf 
    with parent $y$.
    Let $Y_1$ be the resulting rooted tree.
    Note that $Y_1$ has vertex-height at most $h'+1$.
    We set
    \[
        c(X^+,t) = \max\{c_{\ref{thm:main_induction}}(h'+1,Y_1,t), t(2(d-1)-1),1\}.
    \]

    Let $G$ be a graph such that $\tw(G)<t$ and $X^+$ is not a topological minor of $G$.
    Without loss of generality, $G$ is nonnull.
    It follows that there are no $d$ pairwise disjoint topological models of $X'$.
    If $d = 1$, we set $Z$ to be an arbitrary singleton in $V(G)$.
    Otherwise, by \Cref{lemma:helly_with_LCA}, there is a set $Z$ of size at most 
    $t(2(d-1)-1)$ such that $G-Z$ has no topological model of $X'$.
    By possibly adding an arbitrary vertex to $Z$, we assume that $Z \neq \emptyset$.
    Let $\mathcal{F}$ be the family of one-vertex subgraphs of $G-Z$.
    
    We claim that there is no $(Y_1, \mathcal{F})$-structure in $G-Z$.
    Suppose for contradiction that
    there is such a $(Y_1, \mathcal{F})$-structure $\big((U_y \mid y \in V(Y_1)), (P_{yz} \mid yz \in E(\clos(Y_1)))\big)$ in $G-Z$.
    Then, for every $y \in V(Y_0)$, the set $U_y$ is a singleton $\{u_y\}$.
    It follows that $\big((u_y \mid y \in V(Y_0)), (P_{xy} \mid xy \in E(\clos(Y_0)))\big)$ is a topological model of $\clos(Y_0)=X'$, in $G-Z$.
    This contradiction proves that there is no $(Y_1, \mathcal{F})$-structure in $G-Z$.
    
    If $V(G)=Z$, then we get that $|V(G)|= |Z| \leq \max\{t(2(d-1)-1),1\}$ and the result is clear.
    Now suppose $V(G-Z) \neq \emptyset$.

    Let $R$ be an arbitrary singleton in $V(G-Z)$.
    We apply \Cref{thm:main_induction}
    to $h'+1$, $t$, $Y_1$, $G-Z$, $R$, and $\mathcal{F}\vert_{G-(R \cup Z)}$.
    There exist $S \subseteq V(G-Z)$,
    a partition $\mathcal{P}$ of $S$,
    and a tree decomposition $\big(T, (W_x \mid x \in V(T))\big)$ of $(G-Z, \mathcal{P})$ such that, in particular,
    \begin{enumerate}[label={\normalfont\ref{thm:main_induction}.(\Alph*)}]
        \item $V(F) \cap S \neq \emptyset$ for every $F \in \mathcal{F}\vert_{G-(R \cup Z)}$, \label{item:thm:main_induction:hitting:call}
        \item $R \in \mathcal{P}$, \label{item:thm:main_induction:R_in_P:call}
        \item $|P| \leq c_{\ref{thm:main_induction}}(h'+1,Y_1,t)$ for every $P \in \mathcal{P} \setminus \{R\}$, \label{item:thm:main_induction:width_of_P:call}
        \item $|W_x| \leq 2^{h'+1}-2$ for every $x \in V(T)$. \label{item:thm:main_induction:bags_are_small:call}
    \end{enumerate}
    Since $\mathcal{F}$ contains all one-vertex subgraphs of $G-Z$,
    we deduce from \ref{item:thm:main_induction:hitting:call} and \ref{item:thm:main_induction:R_in_P:call} that $S \cup Z = V(G)$.
    Then, $\big(T,(W_x \cup \{Z\} \mid x \in V(T))\big)$ is a tree decomposition of $(G,\mathcal{P} \cup \{Z\})$ 
    of width less than $2^{h'+1}-1$,
    and $|P| \leq  \max\{c_{\ref{thm:main_induction}}(h'+1,Y_1,t), t(2(d-1)-1),1\} = c(X^+,t)$ for every $P \in \mathcal{P}$.
    We conclude by \Cref{obs:partition_to_product_structure} that there is a graph $H$ of treewidth less than $2^{h'+1}-1 \leq 2^{h+1}-1$
    such that $G \subseteq H \boxtimes K_{c(X^+,t)}$.
    Finally, we set $c(X,t) = c(X^+,t)$.
    This proves the theorem.
\end{proof}

\section{Proof of Theorem~\ref{thm:main_Krs}}\label{sec:Krs}

In this section, we prove \Cref{thm:main_Krs}.
The proof is similar to a single step of the induction of 
\Cref{thm:main_induction} (namely from $h=1$ to $h=2$),
except that instead of adding a single vertex to the excluded topological minor,
we add $a$ vertices.
This is inspired by Illingworth, Scott, and Wood's paper~\cite{ISW22}.

\begin{thm}\label{thm:Krs_free_technical}
    Let $a$, $b$, and $t$ be positive integers with $b \geq 2$.
    There exists a positive integer $c_{\ref{thm:Krs_free_technical}}(a,b,t)$
    such that the following holds.
    Let $G$ be a graph of treewidth less than $t$ 
    such that $K_{a,b}$ is not a topological minor of $G$,
    let $a'$ be an integer with $0 \leq a' \leq a+2$,
    and let $R_1, \dots, R_{a'}$ be disjoint subsets of $V(G)$ 
    with $1 \leq |R_i| \leq 4t$ for every $i \in [a']$.
    Then there exist
    a partition $\mathcal{P}$ of $V(G)$
    and a tree decomposition $\big(T,(W_x \mid x \in V(T))\big)$ of $(G,\mathcal{P})$
    such that
    \begin{enumerate}[label={\normalfont(\Alph*)}]
        \item $R_i \in \mathcal{P}$ for every $i \in [a']$, \label{item:Krs:Rs_in_P}
        \item there exists $y \in V(T)$ such that $\{R_i \mid i \in [a']\} \subseteq W_y$, \label{item:Krs:Rs_in_a_Bag}
        \item $|P| \leq c_{\ref{thm:Krs_free_technical}}(a,b,t)$ for every $P \in \mathcal{P}$, \label{item:Krs:parts_are_small}
        \item $|W_x| \leq a+5$ for every $x \in V(T)$. \label{item:Krs:bags_are_small}
    \end{enumerate}
\end{thm}

\begin{proof}
    Let $k = \binom{a+2}{2} \cdot (4t)^{a}$.
    Let $X$ be the rooted star with $a$ leaves.
    By the pigeonhole principle,
    there is a rooted star $Y$
    such that, for every graph $G'$,
    for all families $\mathcal{F}_1, \dots, \mathcal{F}_k$ of connected subgraphs
    of $G'$, 
    if there is a $(Y, \bigcup_{i \in [k]} \mathcal{F}_i)$-structure in $G'$, 
    then there is an $(X, \mathcal{F}_i)$-structure in $G'$
    for some $i \in [k]$.
    Let
    \[
        c_{\ref{thm:Krs_free_technical}}(a,b,t) = \max\{c_{\ref{thm:main_induction}}(2,Y,t), t(2k(b-1)-1),4t\}.
    \]
    Let $R = R_1 \cup \dots \cup R_{a'}$.
    We show by induction on $\big(|V(G)|, |V(G-R)|\big)$ lexicographically that $c_{\ref{thm:Krs_free_technical}}$ witnesses the statement. 
    We call a tuple $(a,b,t,G,(R_i \mid i\in[a']))$ as in the statement an instance.

    If $V(G) \setminus R = \emptyset$, then the result is clear: 
    take $\mathcal{P} = \{R_i \mid i \in [a']\}$ and
    a trivial one-bag tree decomposition of $(G,\mathcal{P})$.
    Now suppose that $V(G) \setminus R \neq \emptyset$.

    If $a'<a+2$, then let $R_{a'+1}$ be an arbitrary singleton
    contained in $V(G) \setminus R$, and apply the induction hypothesis
    on the instance $(a,b,t,G,(R_i \mid i \in [a'+1]))$.
    Thus, we may assume that $a' = a+2$.

    If there is $j \in [a+2]$ such that $N_G(V(G-R)) \cap R_j = \emptyset$,
    then by the induction hypothesis applied to
    the instance
    $\big(a,b,t,G-R_j, (R_i \mid i \in [a+2] \setminus \{j\})\big)$,
    there exist
    a partition $\mathcal{P}_0$ of $V(G) \setminus R_j$
    and a tree decomposition $\big(T_0,(W_x \mid x \in V(T_0))\big)$ of $(G-R_j,\mathcal{P}_0)$
    such that
    \begin{enumerate}[label={\normalfont(\Alph*$'$)}]
        \item $R_i \in \mathcal{P}_0$ for every $i \in [a+2] \setminus \{j\}$,
        \item there exists $y_0 \in V(T_0)$ such that $\{R_i \mid i \in [a+2] \setminus \{j\}\} \subseteq W_{y_0}$,
        \item $|P| \leq c_{\ref{thm:Krs_free_technical}}(a,b,t)$ for every $P \in \mathcal{P}_0$,
        \item $|W_x| \leq a+5$ for every $x \in V(T_0)$.
    \end{enumerate}
    Let $T$ be obtained from $T_0$ by adding a fresh vertex $y$ adjacent to $y_0$,
    and let $W_y = \{R_i \mid i \in [a+2]\}$.
    Then, we get the result for
    \[
        \mathcal{P} = \mathcal{P}_0 \cup \{R_j\}
    \]
    and the tree decomposition $\big(T,(W_x \mid x \in V(T))\big)$ of $(G,\mathcal{P})$.
    Thus, from now on we assume that $N_G(V(G) \setminus R)$ intersects $R_i$ for every $i \in [a+2]$.

    Let $\Lambda$ be the family of all 
    $a$-subsets $\lambda$ of $R$ such that $|\{i \in [a+2] \mid \lambda \cap R_i \neq \emptyset\}|=a$ (i.e. $\lambda$ consists of one vertex of each $R_i$, except for exactly two of these sets, from which it is disjoint).
    Note that $|\Lambda| \leq  \binom{a+2}{2} \cdot (4t)^{a} = k$.
    For every $\lambda \in \Lambda$,
    let $\mathcal{F}_{\lambda}$ be the family of all the connected subgraphs $H$ of $G-R$ such that 
    $\lambda \subseteq N_G(V(H))$.

    \begin{claim}\label{claim:no_K1r_structure}
        There are no $k(b-1)+1$ pairwise disjoint $(Y, \bigcup_{\lambda \in \Lambda} \mathcal{F}_\lambda)$-structures in $G-R$.
    \end{claim}
    \begin{proofclaim}
        Suppose to the contrary that there are $k(b-1)+1$ pairwise disjoint $(Y, \bigcup_{\lambda \in \Lambda} \mathcal{F}_\lambda)$-structures in $G-R$.
        Each of them yields a $(X, \mathcal{F}_\lambda)$-structure
        for some $\lambda \in \Lambda$.
        By the pigeonhole principle,
        there exists $\lambda \in \Lambda$ and $b$ pairwise disjoint
        $(X, \mathcal{F}_\lambda)$-structures
        $\big((U_{j,x} \mid x \in V(X)), (P_{j,xy} \mid xy \in E(\clos(X)))\big)$ for $j \in [b]$ in $G$.
        We denote by $z$ the root of $X$,
        and we fix an arbitrary naming $(x(u) \mid u \in \lambda)$ 
        of the leaves of $X$
        (this is possible since $|\lambda| = a$).
        For each $j \in [b]$,
        let $v_j$ be the unique element in $U_{j,z}$.
        Then, for every $u \in \lambda$, $U_{j,x(u)} \cup V(P_{j,z x(u)})\cup \{u\}$ contains a path $Q_{j,u}$ between $v_j$ and $u$ in $G[U_{j,x(u)} \cup V(P_{j,z x(u)}) \cup \{u\}]$.
        See \Cref{fig:claim_no_structure_K1r}.
        Now, the vertices in $\lambda \cup \{v_j \mid j \in [b]\}$
        together with the paths $Q_{j,u}$ for $j \in [b]$ and $u \in \lambda$
        form a topological model of $K_{a,b}$ in $G$, a contradiction.
    \end{proofclaim}

    \begin{figure}[tp]
        \centering
        \includegraphics{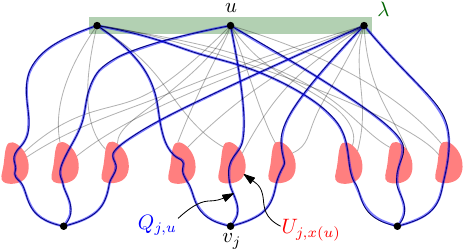}
        \caption{Illustration for the proof of \Cref{claim:no_K1r_structure}. We extract a topological model of $K_{a,b}$ from $b$ disjoint $(X, \mathcal{F}_\lambda)$-structures.}
        \label{fig:claim_no_structure_K1r}
    \end{figure}

    By \Cref{claim:no_K1r_structure} and \Cref{lemma:helly_with_LCA} applied to 
    $G-R$ and the family of all its connected subgraphs that contain a $(Y, \bigcup_{\lambda \in \Lambda} \mathcal{F}_\lambda\vert_{G-R})$-structure, there is a set $Z \subseteq V(G-R)$ with $|Z| \leq t(2k(b-1)-1)$,
    such that
    \begin{enumerate}[label={\normalfont(Z\alph*)}]
        \item there is no $(Y, \bigcup_{\lambda \in \Lambda} \mathcal{F}_\lambda\vert_{G-(R \cup Z)})$-structure in $G-(R \cup Z)$, and \label{Za}
        \item for every component $C$ of $G-(R \cup Z)$,
            $|N_{G-R}(V(C))| \leq 2t$. \label{Zb}
    \end{enumerate}
    In the case that $Z$ is empty, 
    we add to $Z$ an arbitrary vertex of $V(G-R)$. (Recall that $V(G-R) \neq \emptyset$.)  
    Note that this retains both properties: \ref{Za} and \ref{Zb}. 
    Thus, we assume that $Z$ is nonempty.

    If $G-(R \cup Z)$ has no vertices, then the result is clear for $\mathcal{P} = \{R_1, \dots, R_{a+2}, Z\}$ and a single-bag tree decomposition of $(G,\mathcal{P})$.
    Hence, we assume that $G-(R \cup Z)$ is nonnull.

    By \Cref{thm:main_induction} applied to $h=2$, $t$, $Y$, $G-(R \cup Z)$,
    a set $R_0$ being an arbitrary singleton in $V(G-(R \cup Z ))$,
    and the family $\bigcup_{\lambda \in \Lambda} \mathcal{F}_\lambda\vert_{G-(R \cup Z \cup R_0)}$,
    there exist $S \subseteq V(G-(R\cup Z))$ with $R_0 \subseteq S$,
    a partition $\mathcal{P}_0$ of $S$,
    and a tree decomposition $\big(T_0, (W_{0,x} \mid x \in V(T_0))\big)$ of $(G-(R \cup Z), \mathcal{P}_0)$ such that
    \begin{enumerate}[label={\normalfont\ref{thm:main_induction}.\normalfont(\Alph*)}]
        \item $V(F) \cap S \neq \emptyset$ for every $F \in \bigcup_{\lambda \in \Lambda} \mathcal{F}_{\lambda}\vert_{G-(R \cup Z \cup R_0)}$, \label{item:thm:main_induction:hitting:callKrs}
        \setcounter{enumi}{2}
        \item $|P| \leq c_{\ref{thm:main_induction}}(2,Y,t) \leq c_{\ref{thm:Krs_free_technical}}(a,b,t)$ for every $P \in \mathcal{P}_0$ since $|R_0|=1$, \label{item:thm:main_induction:width_of_P:callKrs}
        \item $|W_{0,x}| \leq 2^h-2=2$ for every $x \in V(T_0)$, \label{item:thm:main_induction:bags_are_small:callKrs}
        \setcounter{enumi}{5}
        \item for every component $C$ of $(G-(R \cup Z))-S$, $|N_{G-(R \cup Z)}(V(C))| \leq (2^{h}-2) \cdot 2t = 4t$. \label{item:thm:main_induction:interface:callKrs}
    \end{enumerate}

    Let $\mathcal{C}$ be the family of all the components of $G-(R \cup Z \cup S)$.
    Let $C \in \mathcal{C}$.
    Since $C$ is disjoint from $S \cup Z$,
    by \ref{item:thm:main_induction:hitting:callKrs},
    $C$ does not contain members of $\bigcup_{\lambda \in \Lambda} \mathcal{F}_\lambda$.
    Therefore, there are $i_1(C), i_2(C), i_3(C) \in [a+2]$ distinct such that $N_G(V(C)) \cap R_{i_m(C)} = \emptyset$ for each $m \in [3]$.
    Fix such integers $i_1(C)$, $i_2(C)$, and $i_3(C)$.
    Since $\big(T_0,(W_{0,x} \mid x \in V(T_0))\big)$ is a tree decomposition
    of $(G-(R \cup Z),\mathcal{P}_0)$,
    there exists $x(C) \in V(T_0)$ such that $N_{G-(R \cup Z)}(V(C)) \subseteq \bigcup W_{0,x(C)}$.
    Fix such a vertex $x(C)$.
    Now, let $(R_{C,1}, \dots, R_{C,a'(C)})$ be obtained from $(R_1, \dots, R_{a+2})$
    by removing $R_{i_m(C)}$ for each $m \in [3]$,
    and adding the (at most three) elements of $\{P \cap N_G(V(C)) \mid P \in W_{0,x(C)} \cup \{Z\}\} \setminus \{\emptyset\}$.
    By \ref{Zb}, $|N_G(V(C)) \cap Z| \leq |N_{G-R}(V(D))| \leq 2t \leq 4t$ where $D$ is a component of $G-(R\cup Z)$ containing $C$.
    We obtain that $1 \leq |R_{C,i}| \leq 4t$ for every $i \in [a'(C)]$,
    and $a'(C) \leq a+2$.
    Since $Z \cup S \neq \emptyset$, we have $|V(C)| < |V(G-R)|$.
    Therefore, by the induction hypothesis applied to \[\textstyle G_C = G\left[V(C) \cup \bigcup_{i \in [a'(C)]} R_{C,i} \right],\]
    there exist
    a partition $\mathcal{P}_C$ of $V(G_C)$,
    and a tree decomposition $\big(T_C,(W_{C,x} \mid x \in V(T_C))\big)$ of $(G_C,\mathcal{P}_C)$
    such that
    \begin{enumerate}[label={\normalfont(\Alph*$''$)}]
        \item $R_{C,i} \in \mathcal{P}_C$ for every $i \in [a'(C)]$,
        \item there exists $y(C) \in V(T_C)$ such that $\{R_{C,1}, \dots, R_{C,a'(C)}\} \subseteq W_{C,y(C)}$,
        \item $|P| \leq c_{\ref{thm:Krs_free_technical}}(a,b,t)$ for every $P \in \mathcal{P}_C$, \label{item:Krs:inductioncall:parts_are_small}
        \item $|W_{C,x}| \leq a+5$ for every $x \in V(T_C)$.
    \end{enumerate}

    We assume that the trees $T_0$ and $T_C$ for $C \in \mathcal{C}$ have pairwise disjoint vertex sets.
    For every $x \in V(T_C)$,
    let $W'_{C,x}$ be obtained from $W_{C,x}$ by replacing each boundary part $R_{C,i}\in W_{C,x}$ by the unique element of $\mathcal{P}_0 \cup \{R_1, \dots, R_{a+2}, Z\}$ that contains $R_{C,i}$ and leaving all non-boundary parts unchanged, where $P \in \mathcal{P}_0 \cup \{Z\}$ is such that $P' = P \cap N_G(V(C))$.
    Equivalently,
    \begin{align*}
        W'_{C,x} = 
        \big(W_{C,x} &\setminus \{R_{C,1}, \dots, R_{C,a'(C)}\}\big)
        \ \cup \\
        &\big\{P \in \mathcal{P}_0 \cup \{R_1, \dots, R_{a+2}, Z\} \ \big\vert\ \text{there exists } P' \in W_{C,x} \text{ with } P' \subseteq P\big\}.
    \end{align*}
    Note that $|W'_{C,x}| = |W_{C,x}| \leq a+5$.

    Now, let
    \begin{align*}
        \mathcal{P} &= \{R_1, \dots, R_{a+2}, Z\} \cup \mathcal{P}_0 \cup \bigcup_{C \in \mathcal{C}} \big(\mathcal{P}_C \setminus \{R_{C,i} \mid i \in [a'(C)]\}\big), \\
    \intertext{let $T$ be the tree defined by}
        V(T) &= V(T_0) \cup \bigcup_{C \in \mathcal{C}} V(T_C), \\
        E(T) &= E(T_0) \cup \bigcup_{C \in \mathcal{C}} \big(E(T_C) \cup \{x(C)y(C)\}\big),
    \intertext{and, for every $x \in V(T)$, let}
        W_x &=
        \begin{cases}
            W_{0,x} \cup \{R_1, \dots, R_{a+2}, Z\} & \textrm{if $x \in V(T_0)$,} \\
            W'_{C,x} & \textrm{if $x \in V(T_C)$ for some $C \in \mathcal{C}$.}
        \end{cases}
    \end{align*}

    Clearly $\mathcal{P}$ is a partition of $V(G)$. 
    It remains to show that 
    $\big(T,(W_x \mid x \in V(T))\big)$ is a tree decomposition of $(G,\mathcal{P})$,
    and that \ref{item:Krs:Rs_in_P}-\ref{item:Krs:bags_are_small} are satisfied.

    First we show that $\big(T,(W_x \mid x \in V(T))\big)$ is a tree decomposition of $(G,\mathcal{P})$.
    Let $P \in \mathcal{P}$.
    We show that $\{x \in V(T) \mid P \in W_x\}$ induces a subtree of $T$.
    If $P \in \{R_i \mid i \in [a+2]\} \cup \{Z\} \cup \mathcal{P}_0$,
    then
    \begin{align*}
        \{x \in V(&T) \mid P \in W_x\} = \\
        &\{x \in V(T_0) \mid P \in W_x\} \cup \bigcup_{C \in \mathcal{C}} \ \bigcup_{i \in [a'(C)], R_{C,i} \subseteq P} \{x \in V(T_C) \mid R_{C,i} \subseteq P, R_{C,i} \in W_{C,x}\}.
    \end{align*}
    Since for every $C \in \mathcal{C}$,
    $\big(T_C, (W_{C,x} \mid x \in V(T_C))\big)$ is a tree decomposition
    of $(G_C,\mathcal{P}_C)$, and $\big(T_0, (W_{0,x} \mid x \in V(T_0))\big)$ is a tree decomposition of $(G - (R \cup Z),\mathcal{P}_0)$,
    all these terms which are nonempty induce connected subtrees of $T$.
    Moreover, all the nonempty terms of the form $\{x \in V(T_C) \mid R_{C,i} \subseteq P, R_{C,i} \in W_{C,x}\}$ contain $y(C)$, which is a neighbor of $x(C) \in \{x \in V(T_0) \mid P \in W_x\}$ in $T$.
    We deduce that $\{x \in V(T) \mid P \in W_x\}$ induces a connected subtree of $T$.
    If $P \not \in \{R_i \mid i \in [a+2]\} \cup \{Z\} \cup \mathcal{P}_0$,
    then there exists $C \in \mathcal{C}$ such that $P \in \mathcal{P}_C \setminus \{R_{C,i} \mid i \in [a'(C)]\}$.
    Since $\big(T_C, (W_{C,x} \mid x \in V(T_C))\big)$ is a tree decomposition of $(G_C, \mathcal{P}_C)$, we deduce that $\{x \in V(T_C) \mid P \in W_{C,x}\}$ induces a subtree of $T_C$, and so $\{x \in V(T) \mid P \in W_x\} = \{x \in V(T_C) \mid P \in W_{C,x}\}$ induces a subtree of $T$.

    Consider now an edge $uv \in E(G)$.
    If $u,v \in R \cup Z$, then $u,v \in \bigcup W_x$ for any $x \in V(T_0)$.
    Suppose now that $v\not\in R \cup Z$.
    If $u \in R \cup Z$ and $v \in S$, then, since $\big(T_0, (W_{0,x} \mid x \in V(T_0))\big)$ is a tree decomposition of $(G-(R \cup Z), \mathcal{P}_0)$,
    there is $x \in V(T_0)$ such that $v \in \bigcup W_{0,x}$, 
    and so $u,v \in \bigcup W_{x}$.
    If $u,v \in S$, then, since $\big(T_0, (W_{0,x} \mid x \in V(T_0))\big)$ is a tree decomposition of $(G - (R \cup Z), \mathcal{P}_0)$,
    there is $x \in V(T_0)$ such that $u, v \in \bigcup W_{0,x}$.
    Now suppose that $v \not \in R \cup Z \cup S$.
    Then, there exists $C \in \mathcal{C}$ such that $v \in V(C)$.
    In particular, $uv$ is an edge of $G_C$.
    Since $\big(T_C, (W_{C,x} \mid x \in V(T_C))\big)$ is a tree decomposition
    of $(G_C, \mathcal{P}_C)$, there exists $x \in V(T_C)$ such that $u,v \in \bigcup W_{C,x} \subseteq \bigcup W_x$.
    Up to swapping $u$ and $v$, this shows that for every $uv \in E(G)$,
    there exists $x \in V(T)$ such that $u,v \in \bigcup W_x$.

    Since $\mathcal{P}$ is a partition of $V(G)$, the condition on components of $G-V(G)$ for tree decompositions of $(G,\mathcal{P})$ is vacuously true.

    Altogether, this proves that $\big(T,(W_x \mid x \in V(T))\big)$
    is a tree decomposition of $(G,\mathcal{P})$.
    We now show \ref{item:Krs:Rs_in_P}-\ref{item:Krs:bags_are_small}.

    By the definition of $\mathcal{P}$, \ref{item:Krs:Rs_in_P} holds.
    Moreover, for any $x \in V(T_0)$, $\{R_i \mid i \in [a+2]\} \subseteq W_x$
    and so \ref{item:Krs:Rs_in_a_Bag} holds.
    For every $P \in \mathcal{P}$,
    either $P \in \{R_i \mid i \in [a+2]\}$ and so $|P| \leq 4t \leq c_{\ref{thm:Krs_free_technical}}(a,b,t)$, or
    $P = Z$ and so $|P| \leq t(2k(b-1)-1) \leq c_{\ref{thm:Krs_free_technical}}(a,b,t)$, or
    $P \in \mathcal{P}_0$ and so $|P| \leq c_{\ref{thm:main_induction}}(2,Y,t) \leq c_{\ref{thm:Krs_free_technical}}(a,b,t)$ by \ref{item:thm:main_induction:width_of_P:callKrs}, or
    $P \in \mathcal{P}_C$ for some $C \in \mathcal{C}$, and so
    $|P| \leq c_{\ref{thm:Krs_free_technical}}(a,b,t)$ by \ref{item:Krs:inductioncall:parts_are_small}. 
    Thus,~\ref{item:Krs:parts_are_small} holds.
    In order to prove~\ref{item:Krs:bags_are_small}, consider $x\in V(T)$. 
    If $x\in V(T_0)$ then $|W_x| \leq |W_{0,x}| + (a+3) \leq a+5$ by~\ref{item:thm:main_induction:bags_are_small:callKrs}. 
    If $x\in V(T_C)$ for some $C\in\mathcal{C}$ then $|W_x| = |W'_{C,x}| = |W_{C,x}| \leq a+5$. 
    This proves \ref{item:Krs:bags_are_small}, and concludes the proof of the theorem.
\end{proof}

We apply \Cref{thm:Krs_free_technical} to show \Cref{thm:main_Krs}.

\begin{proof}[Proof of \Cref{thm:main_Krs}]
    If $b=1$, then $K_{a,b}$-topological-minor-free graphs are edgeless,
    and the result is clear for $c(a,b,t) = 1$.
    Now suppose $b \geq 2$.
    Let $c(a,b,t) = c_{\ref{thm:Krs_free_technical}}(a,b,t)$. 
    Let $G$ be a graph with $\tw(G)<t$ that excludes $K_{a,b}$ as a topological minor. 
    We apply \Cref{thm:Krs_free_technical} for $a'=0$. 
    This gives a partition $\mathcal{P}$ of $V(G)$
    and a tree decomposition $\big(T,(W_x \mid x \in V(T))\big)$ of $(G,\mathcal{P})$
    such that, in particular
    \begin{enumerate}[label={\normalfont\ref{thm:Krs_free_technical}.(\Alph*)}]
        \setcounter{enumi}{2}
        \item $|W_x| \leq a+5$ for every $x \in V(T)$. 
    \end{enumerate}
    Therefore, by \Cref{obs:partition_to_product_structure},
    there is a graph $H$ of treewidth less than $a+5$
    such that $G \subseteq H \boxtimes K_{c_{\ref{thm:Krs_free_technical}}(a,b,t)}$.
\end{proof}

\section{Open problem}\label{sec:conclu}

Let $G$ be a graph, $p$ be a positive integer, and $C$ be a set of colors.
A coloring $\varphi\colon V(G)\to C$ of $G$ is \defin{$p$-centered} if for every connected subgraph $H$ of $G$, 
either $\varphi$ uses more than $p$ colors on $H$,
or there is a color that appears exactly once on $H$.
The \defin{$p$-centered chromatic number} of $G$, denoted by \defin{$\chi_p(G)$}, introduced by Nešetřil and Ossona de Mendez~\cite{Nesetril2008}, 
is the least nonnegative integer $k$ such that $G$ admits a $p$-centered coloring using $k$ colors.
\begin{conj}\label{conj:cen_col}
    There is a function $d \colon \mathbb{N} \to \mathbb{N}$
    such that the following holds.
    Let $a$ and $b$ be positive integers with $a \leq b$.
    There exists a positive integer $c$ such that,
    for every positive integer $p$,
    for every graph $G$,
    if $K_{a,b}$ is not a topological minor of $G$, then
    \[
        \chi_p(G) \leq c \cdot p^{d(a)}.
    \]
\end{conj}
We do not know if the conjecture is true even for $a=3$.

\bibliographystyle{plain}
\bibliography{biblio}

\end{document}